\documentclass[fullpage,11pt]{article}

\usepackage[utf8]{inputenc}
\usepackage{amssymb,amsmath,amsthm,color}
\usepackage{faktor}
\usepackage{hyperref}
\usepackage{cleveref}
\usepackage{csquotes}
\usepackage{graphicx}
\usepackage{geometry}
\usepackage{algorithm}
\usepackage[noend]{algpseudocode}

\usepackage{thmtools}
\usepackage{thm-restate}

\theoremstyle{plain}
\newtheorem{theorem}{Theorem}
\newtheorem*{theorem*}{Theorem}

\newtheorem{question}[theorem]{Question}
\newtheorem{conjecture}[theorem]{Conjecture}
\newtheorem{proposition}[theorem]{Proposition}

\theoremstyle{definition}

\theoremstyle{remark}

\title{ On the Number of Shortest Paths in Graphs}

\author{Itai Benjamini, Elad Tzalik}

\begin{document}
\maketitle

\begin{abstract}
    It is proved that the number of shortest paths between two vertices of distance $t$ in a graph with degrees bounded by $\Delta$ is at most $2 \cdot (\frac{\Delta}{2})^t$. This improves upon the na\"ive $\Delta (\Delta-1) ^{t-1}$ bound.
\end{abstract}

Let $G$ be a multigraph with each degree less or equal $\Delta$, and let $t$ be the distance in $G$ between two vertices $x,y$. The aim of this note is to remark that the number of shortest path between $x$ and $y$ is at most exponential in $\frac{\Delta}{2}$, which improve upon the na\"ive $\Delta (\Delta-1) ^{t-1}$ bound.

For $x,y \in V(G)$, let $n_{G}(x,y)$ denote the number of shortest paths between $x$ and $y$ in $G$ and let $d_G(x,y)$ be the distance between $x$ and $y$ in $G$.

\begin{theorem}\label{thm:counting_shortest_paths}
    Let $G$ be a graph of bounded degree $\Delta$. And denote $d_G(x,y)$ by $t$, then $n_G(x,y)\leq \Delta 
 \cdot \left(\lfloor \frac{\Delta}{2} \rfloor \lceil \frac{\Delta}{2} \rceil \right)^{(t-1)/2}$.
\end{theorem}

Our proof is based on analyzing the entropy of a uniformly random shortest $x$ - $y$ path and resembles the entropy proof of the irregular Moore bound by Babu and Radhakrishnan \cite{Babu2010AnEB} \footnote{See \cite{rao_yehudayoff_2020} for a nice presentation of the proof.}. 

\section{Proof of \Cref{thm:counting_shortest_paths}}

    Let $\mathcal{P}$ be a uniformly random $x$-$y$ shortest path. Since the entropy of $\mathcal{P}$, $H(\mathcal{P})$ equals $\log_2(n_G(x,y))$ ( from now on all logarithms are assumed to be with base $2$), it is enough to show that:
    
    \[ H(\mathcal{P}) \leq \log \left( \Delta 
 \cdot \left(\lfloor \frac{\Delta}{2} \rfloor \lceil \frac{\Delta}{2} \rceil \right)^{(t-1)/2} \right).\]

    Without loss of generality assume every edge and vertex in $G$ belong to some shortest path between $x$ and $y$. 
    We can write $\mathcal{P}$ as a vector of the vertices and edges along the path, i.e. 
    
    \[\mathcal{P} = (X_0,E_1,X_1,E_2,\ldots,E_t,X_t).\] Where $X_0=x,X_t=y$, $X_i$ is the $i^{th}$ vertex along the shortest path and $E_i$ is the $i^{th}$ edge on the path. For every vertex $v \in V$ define $N_x(v)=\{e \in E \mid v \in e \text{ and } dist_G(x,e)< d_G(x,v)\}$, and $deg_x(v)= |N_x(v)|$. Similarly define $N_y(v)$ and $deg_y(v)$.

    Notice that the $i^{th}$ edge determines the vertices in location $i-1,i$ hence  $H(X_{i-1} \mid E_i)=0$ and $H(X_{i} \mid E_{i})=0$. Therefore by applying the chain rule (combined with $H(X_0)=0$), we get:
    \begin{equation}\label{eq:chain_rule}
        H(X_0,E_1,X_1\ldots,E_t,X_t) = \sum_{i=1}^t H(E_i \mid X_{i-1},E_{i-1}\ldots,E_1, X_0).
    \end{equation}

    We can also apply the chain rule in the reverse direction (now combining with $H(X_t)=0$) and get:

    \begin{equation}\label{eq:chain_rule_reversed}
        H(X_0,E_1,X_1\ldots,E_t,X_t) = \sum_{i=1}^t H(E_i \mid X_{i+1},E_{i+1}\ldots,E_t, X_t).
    \end{equation}
    
    Observe that since $\mathcal{P}$ is a \emph{shortest} path, the random variable $E_i$ conditioned on the values of $X_{i-1},E_{i-1}\ldots,E_1, X_0$ is identically distributed as the random variable obtained by conditioning $E_i$ only on the value of $X_{i-1}$. 
    This means that $H(E_i \mid X_{i-1},E_{i-1}\ldots,E_1, X_0) = H(E_i \mid X_{i-1})$ and by symmetry $H(E_i \mid X_{i},E_{i+1}\ldots,E_t, X_t) = H(E_i \mid X_{i})$. 
    By plugging this into \Cref{eq:chain_rule,eq:chain_rule_reversed} and adding them up we get:

    \begin{align}\label{eq:chain_rule_both_directions}
    2H(\mathcal{P}) &= \sum_{i=1}^t \left( H(E_i \mid X_i)+H(E_i \mid X_{i-1}) \right) \\  
    &= \left(H(E_1 \mid X_0) + H(E_t \mid X_t) \right) + \sum_{i=1}^{t-1} H(E_i \mid X_i) + H(E_{i-1} \mid X_i).
    \end{align}

    We now bound each term in the last expression.
    For the first term, we note that $H(E_1 \mid X_0)=H(E_1) \leq \log(\Delta)$ since $X_0=x$ and $E_1$ is supported on at most $\Delta$ elements. Similarly $H(E_t \mid X_t) \leq \log(\Delta)$.

    For the sum, notice that by definition $H(E_{i+1} \mid X_i) = \sum_{v\in V} \Pr[X_i=v] \cdot H(E_{i+1} \mid X_i=v)$. Observe that since $\mathcal{P}$ is a shortest path, given that $X_{i}=v$ the support of the edge $E_{i+1}$ is contained in $N_y(v)$ and the support of $E_i$ is contained in $N_x(v)$. Hence:

    \begin{align}\label{eq:expanding_terms_in_entropy_sum}
        H(E_{i+1} \mid X_i) + H(E_i \mid X_i) &= \sum_{v} \Pr[X_i=v] \cdot 
 \left( H(E_{i+1} \mid X_i=v) + H(E_i \mid X_i=v) \right) \\  
 &\leq \sum_{v} \Pr[X_i=v] \cdot 
 \left( \log( deg_x(v) ) + \log (deg_y(v)) \right) \\  
 &\leq  \sum_{v} \Pr[X_i=v] \cdot 
 \log( deg_x(v) deg_y(v)).
    \end{align}
    
    Note that $N_x(v),N_y(v)$ are disjoint as an edge in the intersection will imply that there is a vertex appearing twice along a shortest path which is absurd. Hence since the graph is of degree bounded by $\Delta$ we have that $deg_x(v)+deg_y(v)\leq \Delta$, hence:

    \begin{equation}\label{eq:log_degree_bound}
        \log( deg_x(v) deg_y(v)) \leq \log \left( \lfloor \frac{\Delta}{2} \rfloor \cdot \lceil \frac{\Delta}{2} \rceil \right).
    \end{equation} 

    Plugging \Cref{eq:log_degree_bound} into \Cref{eq:expanding_terms_in_entropy_sum} we get
    
    \[H(E_{i+1} \mid X_i) + H(E_i \mid X_i) \leq \log \left( \lfloor \frac{\Delta}{2} \rfloor \cdot \lceil \frac{\Delta}{2} \rceil \right).\]

    Finally, by plugging everything into \Cref{eq:chain_rule_both_directions}:

    \[ 2\log(n_G(x,y)) = 2H(\mathcal{P})\leq 2 \log(\Delta) +(t-1) \log \left( \lfloor \frac{\Delta}{2} \rfloor \cdot \lceil \frac{\Delta}{2} \rceil \right). \]

    Hence $H(\mathcal{P}) \leq \log \left( \Delta 
 \cdot \left(\lfloor \frac{\Delta}{2} \rfloor \lceil \frac{\Delta}{2} \rceil \right)^{(t-1)/2} \right)$ as needed.

\section{Comments}

\paragraph{A tight example}
    Consider the cycle graph on $2t$ nodes, $C_{2t}$ and replace each edge by  $ \lfloor \frac{\Delta}{2} \rfloor$ multiple edges or  $ \lceil \frac{\Delta}{2} \rceil$ multiple edges in an alternating fashion. Denote this graph by $C_{2t,\Delta}$. This multigraph achieves the bound of \Cref{thm:counting_shortest_paths} when $\Delta$ is even, as well as when $\Delta$ and $t$ are odd.

We believe that the bound in \Cref{thm:counting_shortest_paths} is not optimal in case $\Delta$ is odd and $t$ is even, and can be strengthened as follows:

\begin{conjecture}
    In the case of odd $\Delta$ and even $t$ the number of shortest paths agree with the graph $C_{2t,\Delta}$. Explicitly, the number of shortest paths is bounded by $2 \left(\lfloor \frac{\Delta}{2} \rfloor \lceil \frac{\Delta}{2} \rceil \right)^{t/2}$ .
\end{conjecture}

\paragraph{Random walk on graphs} Let $G$ be a weighted graph and consider a random walk on $G$. By approximating the weights with rationals in $\frac{1}{\Delta} \mathbb{Z}$ we can approximate the walk on this graph by a walk on a multigraph $G'$. For two vertices $x,y$, let $t$ be the minimal integer for which a walk according to $G'$ starting at $x$ reaches $y$ in $t$ steps, then in such case the probability to reach $y$ from $x$ is $\frac{n_{G'}(x,y)}{\Delta^t} \leq (\frac{1}{2})^{t-1}$. By letting $\Delta$ go to $0$ we have:

\[ \Pr[ \text{A random walk from } x \text{ to } y \text{ takes } t \text{ steps} ] \leq \left(\frac{1}{2}\right)^{t-1}\]

\paragraph{Simple graphs} Note that a slightly stronger bound applies to simple graphs. In case $t = 1$ we have $n_G(x,y)\leq 1$ since simple graphs have no parallel edges. In case $t=2$ one has $n_G(x,y)\leq \Delta$ since again, the collection of $2$-paths from $x$ to $y$ must have disjoint midpoints by simplicity. For $t \geq 3$, one gets a bound sharper by $\approx \frac{\Delta}{4}$ factor. This is obtained by observing that $H(E_1 \mid X_1)=0$ (there's a unique edge going from $X_1$ to $X_0=x$ by simplicity), hence $H(E_1 \mid X_1)+ H(E_2 \mid X_1) \leq \log(\Delta-1)$. Applying this improved bound to the first and last terms of the sum \Cref{eq:chain_rule_both_directions} gives: $n_G(x,y) \leq \Delta (\Delta-1) \cdot \left(\lfloor \frac{\Delta}{2} \rfloor \lceil \frac{\Delta}{2} \rceil \right)^{(t-3)/2}$. 

One obtains a simple graphs achieving the above bound by again taking the cycle on $2t$ vertices, and replacing each vertex by $\frac{\Delta}{2}$ vertices, and replacing each edge by a bi-clique.
\footnote{In case $\Delta$ is odd, replace each vertex by $ \lceil \frac{\Delta}{2} \rceil$ vertices and each edge by either a bi-clique or a biregular graph of degree $ \lfloor \frac{\Delta}{2} \rfloor$ in an alternating fashion.}
Let $G$ be the aforementioned graph, then the number of shortest paths between two vertices $u,v$ that correspond to antipodal points in $C_{2t}$ is $\left(\frac{\Delta}{2} \right)^{t-1}$.

\begin{question}
    What is the exact maximum number $n_G(x,y)$ among all simple graphs in terms of the bound $\Delta$ on the degree, and the distance $t$ between $x$ and $y$.
\end{question}


\paragraph{Large girth}   It is plausible to suspect that assuming large girth one may get a much better upper bound. While large girth does imply an improved bound, it improves the exponent, but not the base of the exponent. If a graph $G$ has girth strictly bigger than $g$ then one can shave off the exponent by roughly $\Theta(g)$ (additive). This follows similarly to the analysis in the case of simple graphs, by noticing that each vertex $v$ in the $\frac{g}{2}$ neighborhood around $x$ has $deg_x(v)=1$ (otherwise a cycle is formed by $v$, two elements in $N_x(v)$ and the paths from them towards $x$) hence the entropy bounds in $\Cref{eq:log_degree_bound}$ are sharper for the edges lying in the $\frac{g}{2}$ neighborhoods of $x,y$. A construction of a graph with $(\frac{\Delta}{2})^{t-\Theta(g)}$ can be made similarly to the simple graph construction: Instead of replacing each edge of a cycle of length $2t$ by a bi-clique with $\frac{\Delta}{2}$ vertices on each side one replaces each edge by a bi-regular graph of degree $\frac{\Delta}{2}$ and large girth.

There are some natural examples of graph families where one may get an improvement in the base of the exponent beyond $\frac{\Delta}{2}$. E.g consider the family of triangulation of the two dimensional sphere. We claim that in such a case one has an improved bound:

\begin{proposition}
    Let $G$ be a triangulation of the two dimensional sphere with bounded degree $\Delta$, and let $x,y$ be two vertices of distance  $t>1$, then $n_G(x,y)\leq \Delta 
 \cdot \left(\lfloor \frac{\Delta-2}{2} \rfloor \lceil \frac{\Delta-2}{2} \rceil \right)^{(t-1)/2}$.
\end{proposition}

\begin{proof}
    We claim that $\forall v \neq x,y$ we have $deg_x(v)+deg_y(v) \leq \Delta-2$. Notice that combining this with the bound in \Cref{thm:counting_shortest_paths} yields the claim. 
    
    Let $N$ be the set of neighbors of $x$, and assume $v \in N$. It is well known (e.g. can be proved by excision) that the graph induced on $N$ is a cycle (by excision it has the homology of the cycle and hence it is a cycle), therefore $v$ has two edges to neighbors in $N$ which are not in $N_x(v) \cup N_y(v)$, hence $\mid N_x(v) \cup N_y(v) \mid \leq \Delta-2$ as needed. By contracting $N$ towards $x$ the same inequality holds for the $2$-neighbourhoods of $x$, and by induction for all $v \neq x,y$ that lie along a path from $x$ to $y$.
\end{proof}

Finally we believe it is interesting to prove higher dimensional analogues of \Cref{thm:counting_shortest_paths}. Let $X$ be a $d$-dimensional simplicial complex such that each $d-1$ face is contained in at most $\Delta$ different $d$-faces. A cycle $c \in Z_{d-1}(X, \mathbb{F}_2)$ is  \emph{irreducible} iff whenever $c = c_1+c_2$ with $c_1$, $c_2$ of disjoint support then either $c_1=0$ or $c_2=0$. The minimum filling size of $c$ is defined to be $m(c) = \min \{ \mid supp(f) \mid :  f \in C_d(X,\mathbb{F}_2) , \partial f = c\}$ and each element achieving the minimum is said to be a minimal filling of $c$  \footnote{For related definitions see \cite{MeshulamW09}} :

\begin{conjecture}
     There exist a constant $c(d)<1$ such that there are at most  $O \left( (c(d) \Delta)^{m(c)} \right)$ minimal fillings of $c$.
\end{conjecture}

Consider $\mathbb{Z}^3$ with the cubical structure and give each $2$-face $\ell$ a weight $X_{\ell}$ which is a Bernulli with probability $p$, independently. 
\begin{question}
    What is the typical number of minimal fillings of the square of side length $n$? Show that it is exponential in $n$.
\end{question}
For a one dimension analogue see \cite{FirstPassage}.

\paragraph{Acknowledgments} We thank Asaf Petruschka for finding a mistake in a previous version. I.B thanks the Israel Science Foundation for their support.  

\bibliographystyle{alphaurl}
\bibliography{References.bib}

\end{document}